\documentclass[reqno,11pt,a4paper]{amsart}
\usepackage{graphics,epic}
\usepackage{amsmath,amssymb, amsthm}
\usepackage[all,2cell]{xy}
\usepackage{color}
\usepackage{hyperref}

\numberwithin{equation}{section}
\theoremstyle{plain}
\newtheorem{thm}{Theorem}[section]
\newtheorem*{thm*}{Theorem}

\newtheorem{cor}[thm]{Corollary}
\newtheorem*{cor*}{Corollary}
\newtheorem{lem}[thm]{Lemma}
\theoremstyle{definition}
\newtheorem{defn}[thm]{Definition}
\newtheorem{exm}[thm]{Example}
\theoremstyle{remark}
\newtheorem{rmk}[thm]{\bf Remark}

\renewcommand{\Im}{\mathrm{Im}}

\hypersetup{hidelinks,
	colorlinks=true,
	allcolors=black,
	pdfstartview=Fit,
	breaklinks=true}

\begin{document}
\title{N-recollements and Virtually Gorenstein Algebras}

\author{Dawei Shen}
\address{Dawei Shen\\
         School of Mathematics and Statistics\\
         Henan University\\
         Kaifeng \\
         Henan 475004 \\
         China}
\email{sdw12345@mail.ustc.edu.cn}

\author{Hao Su*}
\address{Hao Su\\
         School of Science\\
         Tianjin Chengjian University\\
         Tianjin 300384 \\
         China}
\email{suhaomt@tcu.edu.cn}

\thanks{*Corresponding author}
\date{\today}

\begin{abstract}
    The relation between the $n$-recollements of stable categories of Gorenstein projective modules and the virtual Gorensteinness of algebras are investigated. Let $A,B$, and $C$ be finite dimensional algebras.  We prove that if the stable category of Gorenstein projective $A$-modules admits a $n$-recollement relative to that of $B$ and $C$ with $n\geq 2$, then $A$ is virtually Gorenstein if and only if so are $B$ and $C$.
\end{abstract}

\maketitle

\section{Introduction}

   Recollements of triangulated categories were introduced by Beilinson, Bernstein, and Deligne \cite{BEILINSON1988317} on perverse sheaves and play an important role in algebraic geometry and representation theory now. Roughly speaking, a recollement allows us to decompose a triangulated category into two relatively smaller ones. This technique is very useful in the study of derived categories, especially in reducing homological conjectures and computing homological invariants.

   The concept of Gorenstein projective module goes back to the work of Auslander and Bridger \cite{Auslander1969}. Since then, they are central in relative homological algebra. Virtually Gorenstein algebras introduced in \cite{Beligiannis2005} are a natural generalization of Gorenstein algebras and share many properties with genuine Gorenstein algebras. The virtual Gorensteinness provides a useful tool for the study of the Gorenstein Symmetry Conjecture.

   In \cite{Qin2016}, Qin and Han introduced the concept of $n$-recollement of triangulated categories and clarified the relation between  $n$-recollements of derived categories of algebras and the Gorensteinness of algebras. Inspired by their work, we investigate the relation between the $n$-recollements of stable categories of Gorenstein projective modules and the virtual Gorensteinness of algebras. Let $\mathcal{A}$ be an abelian category with enough projective objects. We denote by ${\rm{GP}}(\mathcal{A})$ the full subcategory of $\mathcal{A}$ which is formed by all Gorenstein projective objects and  $\underline{{\rm{GP}}}(\mathcal{A})$ the stable category of ${\rm{GP}}(\mathcal{A})$.

   Our main result is as follows.

   \begin{thm*}[I]
        Let $A,B$ and $C$ be finite dimensional algebras, and $\underline{\rm{GP}}(A)$ admits an $n$-recollement relative to $\underline{\rm{GP}}(B)$ and $\underline{\rm{GP}}(C)$ which restricts to finite dimensional modules.
        Then
        \begin{enumerate}
          \item $n=1$: if $A$ is virtually Gorenstein then so are $B$ and $C$;
          \item $n\geq 2$: $A$ is virtually Gorenstein if and only if so are $B$ and $C$.
        \end{enumerate}
   \end{thm*}

    We also study the relation between $n$-recollements of derived categories of algebras and the virtual Gorensteinness of algebras. We obtain the following result.

    \begin{thm*}[II]
        Let $A,B$ and $C$ be finite dimensional algebras, and  $\mathcal{D}(\mathrm{Mod}A)$ admits a $n$-recollement relative to  $\mathcal{D}(\mathrm{Mod} B)$ and $\mathcal{D}(\mathrm{Mod} C)$. Then
        \begin{enumerate}
          \item $n=5$: if $A$ is virtually Gorenstein then so are $B$ and $C$;
          \item $n\geq 6$: $A$ is virtually Gorenstein if and only if so are $B$ and $C$.
        \end{enumerate}
    \end{thm*}

    As an application of Theorem ($\textrm{II}$), we reprove that virtually Gorensteinness is invariant under derived equivalence.

   The paper is organized as follows. In section 2 we will recall some basic facts of $n$-recollements. In Section 3, we study the relation between $n$-recollements of stable categories of Gorenstein projective modules and the virtual Gorensteinness of algebras and prove Theorem ($\textrm{I}$). In Section 4, we investigate the relation between $n$-recollements of derived categories of algebras and the virtual Gorensteinness of algebras and prove Theorem ($\textrm{II}$).

\section{$n$-recollements of triangulated categories}
    In this section, we review some of the standard facts of $n$-recollements of triangulated categories.

    \begin{defn}[\cite{BEILINSON1988317}]
      Let $\mathcal{D},\mathcal{D}'$ and $\mathcal{D}''$ be triangulated categories. A {\em recollement} of $\mathcal{D}$ relative to $\mathcal{D}'$ and $\mathcal{D}''$ is a diagram
      $$
         \xymatrix@C=6em{
                    \mathcal{D}'\ar@{->}[r]|-{i_{\ast}=i_{!} }& \mathcal{D} \ar@<2ex>[l]^{i^{!}} \ar@{->}[r]|-{j^{!}=j^{\ast}} \ar@<-2ex>[l]_{i^{\ast}} & \mathcal{D}''\ar@<2ex>[l]^{j_{\ast}}\ar@<-2ex>[l]_{j_{!}}\\
         }
      $$
      of triangulated categories and triangle functors such that
      \begin{enumerate}
        \item $(i^{\ast},i_{\ast}),(i_{!},i^{!}),(j_{!},j^{!})$ and $(j^{\ast},j_{\ast})$ are adjoint pairs;
        \item $i_{\ast},j_{!}$ and $j_{\ast}$ are full embeddings;
        \item $j^{!}i_{\ast}=0$ (and thus also $i^{!}j_{\ast}=0$ and $i^{\ast}j_{!}=0$);
        \item for each $X\in D$, there are triangles
        $$
            i_{!}i^{!}(X)\rightarrow X \rightarrow j_{\ast}j^{\ast}(X) \rightarrow i_{!}i^{!}(X)[1]
        $$
        $$
            j_{!}j^{!}(X)\rightarrow X \rightarrow i_{\ast}i^{\ast}(X) \rightarrow j_{!}j^{!}(X)[1]
        $$
      \end{enumerate}
      where the arrows are given by adjunctions.
    \end{defn}

    Given a triangulated category $\mathcal{T}$ with small coproducts, an object $X\in \mathcal{T}$ is {\em compact} if $\rm{Hom}_{\mathcal{T}}(X,-)$ preserves small coproducts. We denote by $\mathcal{T}^{c}$ the full subcategory of $\mathcal{T}$ consisting of all compact objects. In a recollement, the functors in the first row preserve compact objects.
    For a set $\mathcal{S}$ of objects of $\mathcal{T}$, we denote by $\mathrm{Loc}\, \mathcal{S}$ the smallest triangulated subcategory of $\mathcal{T}$ containing $\mathcal{S}$ and closed under taking coproducts. A triangulated category $\mathcal{T}$ with small coproducts is called {\em compactly generated} if there is a set  $\mathcal{S}$ of compact objects such that $\mathcal{T}=\mathrm{Loc}\,\mathcal{S}$.

    \begin{lem}[{\cite[Theorem 2.1]{ASENS_1992_4_25_5_547_0}}]\label{Neeman92}
    Let $\mathcal{D},\mathcal{D}'$ and $\mathcal{D}''$ be compactly generated and
            $$
                \xymatrix@C=3em{
                    \mathcal{D}'\ar@{->}[r]|-{i_{\ast}=i_{!} }& \mathcal{D} \ar@<2ex>[l]^{i^{!}} \ar@{->}[r]|-{j^{!}=j^{\ast}} \ar@<-2ex>[l]_{i^{\ast}} & \mathcal{D}''\ar@<2ex>[l]^{j_{\ast}}\ar@<-2ex>[l]_{j_{!}}\\
            }
            $$
        be a recollement of $\mathcal{D}$ relative to $\mathcal{D}'$ and $\mathcal{D}''$.
        Then the first row restricts to compact objects and the restriction of $i^*$ is dense up to direct summand.
    \end{lem}

    \begin{lem}[{\cite[Proposition 2.5]{Liu2015}}]\label{LiuP2.5}
        Let
            $$
                \xymatrix@C=3em{
                    \mathcal{D}'\ar@{->}[r]|-{i_{\ast}=i_{!} }& \mathcal{D} \ar@<2ex>[l]^{i^{!}} \ar@{->}[r]|-{j^{!}=j^{\ast}} \ar@<-2ex>[l]_{i^{\ast}} & \mathcal{D}''\ar@<2ex>[l]^{j_{\ast}}\ar@<-2ex>[l]_{j_{!}}\\
            }
            $$
        be a recollement of $\mathcal{D}$ relative to $\mathcal{D}'$ and $\mathcal{D}''$.
        Let $\mathcal{T}$ is a thick subcategory of $\mathcal{D}$ and $\mathcal{T}'=i^{\ast}\mathcal{T}$ and $\mathcal{T}''=j^{\ast}\mathcal{T}$. If $i_{\ast}i^{\ast}\mathcal{T} \subset \mathcal{T}$ and $j_{\ast}j^{\ast}\mathcal{T}\subset \mathcal{T}$, then there exists a recollement of triangulated quotient categories
            $$
                \xymatrix@C=3em{
                    \mathcal{D}'/\mathcal{T}' \ar@{->}[r]|-{\tilde{i_{\ast}}=\tilde{i_{!}} }& \mathcal{D}/\mathcal{T} \ar@<2ex>[l]^{\tilde{i^{!}}} \ar@{->}[r]|-{\tilde{j^{!}}=\tilde{j^{\ast}}} \ar@<-2ex>[l]_{\tilde{i^{\ast}}} & \mathcal{D}''/\mathcal{T}'' \ar@<2ex>[l]^{\tilde{j_{\ast}}}\ar@<-2ex>[l]_{\tilde{j_{!}}}\\
            }
            $$
    \end{lem}

    \begin{defn}[{\cite[Definition 2]{Qin2016}}]
        Let $\mathcal{D},\mathcal{D}'$ and $\mathcal{D}''$ be triangulated categories and $n$ a positive integer. An {\em n-recollement}  of $\mathcal{D}$ relative to $\mathcal{D}'$ and $\mathcal{D}''$ is a diagram with $n+2$ rows
        $$
         \xymatrix@C=6em{
                    \mathcal{D}'\ar@<1ex>[r]  \ar@<-1ex>[r]_-{\vdots} & \mathcal{D} \ar@<-2ex>[l]  \ar@{->}[l]  \ar@<1ex>[r] \ar@<-1ex>[r]_-{\vdots}  & \mathcal{D}'' \ar@<-2ex>[l]  \ar@{->}[l]   \\
         }
        $$
        of triangulated categories and triangle functors such that every consecutive three rows form a recollement.
    \end{defn}

    It is obvious that 1-recollement is just a recollement. Moreover, if $\mathcal{D}$ admits an $n$-recollement $R$ relative to $D'$ and $D''$, then any consecutive $m$ rows in $R$ form a $(m-2)$-recollement for $3\leq m \leq n$. $N$-recollements are also referred as ladders of height $n$ in some other references (see \cite[Section 3]{Hugel2017}).

    We focus on $n$-recollements of derived categories of algebras. For convenience, we fix a field $k$. Let $A$ be a finite dimensional $k$-algebra. We consider the category Mod$A$ of left $A$-modules and its full subcategories mod$A$, Proj$A$, and proj$A$ of finitely generated modules, projective modules, and finitely generated projective modules respectively. Let $\mathcal{A}$ be any above subcategory of Mod$A$, for $\ast \in \{\emptyset,b\}$, we denote by $\mathcal{D}^{\ast}(\mathcal{A})$ (resp. $\mathcal{K}^{\ast}(\mathcal{A})$) the derived category (resp. homotopy category) of cochain complexes of objects in $\mathcal{A}$ satisfying the corresponding boundedness condition.

    The following lemma will be used in the proof of the Theorem \ref{thm2}.

    \begin{lem}[{\cite[Lemma 2.7]{Hugel2017}},{\cite[Lemma 2.14]{Cummings2023}}]\label{P2.3}
        Let $A$ and $B$ be two finite dimensional algebras, and $F:\mathcal{D}({\rm{Mod}}A) \to \mathcal{D}({\rm{Mod}}B)$ be a triangle functor with a right adjoint $G$. Consider the following conditions
        \begin{enumerate}
          \item $F$ preserves $\mathcal{K}^{b}({\rm{proj}})$;
          \item $G$ admits a right adjoint;
          \item $G$ preserves $\mathcal{D}^{b}({\rm{mod}})$;
          \item $G$ preserves $\mathcal{D}^{b}({\rm{Mod}})$;
          \item $F$ preserves $\mathcal{K}^{b}(\rm{Proj})$.
        \end{enumerate}
        Then we have $(1)\Leftrightarrow(2)\Leftrightarrow (3)\Rightarrow (4)\Leftrightarrow(5)$.
    \end{lem}

\section{$n$-recollements of stable categories}

    In this section, we will recall the definition of virtually Gorenstein algebras and prove Theorem (I).

    Let $\mathcal{A}$ be an abelian category with enough projective objects. We denote by P($\mathcal{A}$) the full subcategory of  $\mathcal{A}$ which is formed by all projective objects.

    \begin{defn}[\cite{Auslander1969}]\label{d3.1}
      Let $\mathcal{A}$ be an abelian category with enough projective objects. An object $M\in \mathcal{A}$ is called $Gorenstein\ projective$ if there is a complex $P^{\bullet}$
      $$
        \xymatrix@C=3em{
        \cdots \ar@[r][r] & P^{-1} \ar@{->}[r]^{d^{-1}} &P^{0} \ar@{->}[r]^{d^{0 }} &P^{1} \ar@{->}[r]^{d^{1}} &\cdots
        }
      $$
      of objects in $\rm{P}(\mathcal{A})$ such that
      \begin{enumerate}
        \item $P^{\bullet}$ is acyclic;
        \item for all $Q \in {\rm{P}}(\mathcal{A})$, the complex $\mathcal{H}om_{\mathcal{A}}^{\bullet}(P^{\bullet},Q)$ is acyclic;
        \item $M \simeq \Im d^{0}$.
      \end{enumerate}
    \end{defn}

     We denote by GP($\mathcal{A}$) the full subcategory of $\mathcal{A}$ which is formed by all Gorenstein projective objects. The GP($\mathcal{A}$) is a Frobenius category whose projective objects are exactly the projective-injective objects in $\mathcal{A}$. The stable category $\underline{\rm{GP}}(\mathcal{A})$ is a triangulated category. There is a triangle embedding $\underline{\rm{GP}}(\mathcal{A}) \hookrightarrow \mathcal{D}^{b}_{sg}(\mathcal{A})=\mathcal{D}^{b}(\mathcal{A})/\mathcal{K}^{b}({\rm{P}}(\mathcal{A}))$ (\cite{Buchweitz2021},\cite{Happel1991}).

    Virtually Gorenstein algebras was introduced in \cite{Beligiannis2005}. They share many properties with genuine Gorenstein algebras. This class of algebras has rich homological structures and is closed under various operations such as derived equivalences and stable equivalences of Morita type.

    \begin{defn}[{\cite{Beligiannis2005}}]
        Let $A$ be a finite dimensional algebra. Then $A$ is $virtually$ $Gorenstein$ if for $X\in {\rm{Mod}}A$, the functor ${\rm{Ext}}_{A}^{i}(X,-)$ vanishes for all $i>0$ on all Gorenstein injective modules iff ${\rm{Ext}}_{A}^{i}(-,X)$ vanishes for all  $i>0$  on all Gorenstein projective modules.
    \end{defn}

    Not only Gorenstein algebras, but radical square zero algebras and algebras of finite representation type are all virtually Gorenstein algebras. We study the virtual Gorensteinness of algebra through $n$-recollements.

    Let $A$ be a finite dimensional algebra, we  abbreviate GP(Mod$A$) and $\underline{\rm{GP}}({\rm{Mod}}A)$ as GP($A$) and $\underline{\rm{GP}}(A)$ respectively. Following \cite{Beligiannis2005}, the triangulated category $\underline{\rm{GP}}(A)$ is compactly generated.

    Let ${\rm{Gp}}(A)={\rm{GP}}(A) \cap {\rm{mod}} A$ and $\underline{\rm{Gp}}(A)$ be the stable category of ${\rm{Gp}}(A)$. There is a full triangle embedding $ \underline{\rm{Gp}}(A) \hookrightarrow (\underline{\rm{GP}}(A))^{c}$.

    The following lemma provides a equivalent characterization of virtually Gorenstein algebra.

    \begin{lem}[{\cite[Theorem 8.2]{Beligiannis2005}}]\label{B05T8.2}
        Let $A$ be a finite dimensional algebra, $A$ is $virtually$ $Gorenstein$ if and only if $(\underline{\rm{GP}}(A))^{c}= \underline{\rm{Gp}}(A)$.
    \end{lem}

    \begin{thm}\label{thm1}
        Let $A,B$ and $C$ be finite dimensional algebras, and $\underline{\rm{GP}}(A)$ admits an $n$-recollement relative to $\underline{\rm{GP}}(C)$ and $\underline{\rm{GP}}(B)$ which restricts to $\underline{\rm{Gp}}$.
        \begin{enumerate}
          \item $n=1$: if $A$ is virtually Gorenstein then so are $B$ and $C$;
          \item $n\geq 2$: $A$ is virtually Gorenstein if and only if so are $B$ and $C$.
        \end{enumerate}
   \end{thm}

   \begin{proof}
   (1) Given a recollement
            \begin{equation*}
                \xymatrix@C=3em{
                  \underline{\rm{GP}}(C) \ar@{->}[r]   & \underline{\rm{GP}}(A) \ar@<-1ex>[l]  \ar@<1ex>[l]  \ar@{->}[r]   & \underline{\rm{GP}}(B) \ar@<-1ex>[l]  \ar@<1ex>[l]   \\
                }
        \end{equation*}
        Under conditions in the theorem, we have the following diagram.
        \begin{equation*}
            \xymatrix@C=3em{
                \underline{\rm{GP}}(C) \ar@{->}[r]|-{i_{\ast}=i_{!}}              & \underline{\rm{GP}}(A) \ar@<-1ex>[l]_{i^{\ast}} \ar@<1ex>[l]^{i^{!}}  \ar@{->}[r]|-{j^{!}=j^{\ast}}            & \underline{\rm{GP}}(B) \ar@<-1ex>[l]_{j_{!}}  \ar@<1ex>[l]^{j_{\ast}}   \\
                (\underline{\rm{GP}}(C))^{c}  \ar@{->}[u]           & (\underline{\rm{GP}}(A))^{c}  \ar@{->}[l]_{i^{\ast}} \ar@{->}[u]     & (\underline{\rm{GP}}(B))^{c} \ar@{->}[l]_{j_{!}}  \ar@{->}[u]              \\
                \underline{\rm{Gp}}(C) \ar@{->}[r]|-{i_{\ast}=i_{!}}  \ar@{->}[u] & \underline{\rm{Gp}}(A) \ar@<-1ex>[l]_{i^{\ast}} \ar@<1ex>[l]^{i^{!}}  \ar@{->}[r]|-{j^{!}=j^{\ast}} \ar@{->}[u]      & \underline{\rm{Gp}}(B) \ar@<-1ex>[l]_{j_{!}}  \ar@<1ex>[l]^{j_{\ast}} \ar@{->}[u]  \\
            }
        \end{equation*}
        Assume $A$ is virtually Gorenstein. Then $(\underline{\rm{GP}}(A))^{c}=\underline{\rm{Gp}}(A)$. For any object $Z\in (\underline{\rm{GP}}(B))^{c}$, we have and $Z \cong j^{!}j_{!}Z $ in $\underline{\rm{GP}}(B)$. Since $ j_{!}Z\in (\underline{\rm{GP}}(A))^{c}=\underline{\rm{Gp}}(A)$, we have $j^{!}j_{!}Z \in \underline{\rm{Gp}}(B)$. Then $(\underline{\rm{GP}}(B))^{c}=\underline{\rm{Gp}}(B)$. By Lemma \ref{B05T8.2}, $B$ is a virtually Gorenstein algebra.

       For any object $Y\in (\underline{\rm{GP}}(C))^{c}$, since $\underline{\rm{GP}}(\rm{Mod})$ is compactly generated, by Lemma \ref{Neeman92} there is $X\in (\underline{\rm{GP}}(A))^{c}=\underline{\rm{Gp}}(A)$ such that
       $Y$ is isomorphic to a direct summand of $i^{\ast} X$. Since  $i^{\ast}$ restricts to $\underline{\rm{Gp}}$,
       then $Y\in \underline{\rm{Gp}}(C)$. By Lemma \ref{B05T8.2}, $C$ is a virtually Gorenstein algebra.

      (2) Similarly, we have the following diagram.
      \begin{equation*}
        \xymatrix@C=3em{
            \underline{\rm{GP}}(C) \ar@<1ex>[r]|-{i_{\ast}=i_{!}} \ar@<-1ex>[r]              & \underline{\rm{GP}}(A) \ar@<-2ex>[l]_{i^{\ast}} \ar@{->}[l]|-{i^{!}}  \ar@<1ex>[r]|-{j^{!}=j^{\ast}} \ar@<-1ex>[r]             & \underline{\rm{GP}}(B) \ar@<-2ex>[l]_{j_{!}}  \ar@{->}[l]|-{j_{\ast}}   \\
            (\underline{\rm{GP}}(C))^{c} \ar@<-1ex>[r]_{i_{\ast}} \ar@{->}[u]           & (\underline{\rm{GP}}(A))^{c}  \ar@<-1ex>[l]_{i^{\ast}} \ar@<-1ex>[r]_{j^{!}} \ar@{->}[u]                     & (\underline{\rm{GP}}(B))^{c} \ar@<-1ex>[l]_{j_{!}}  \ar@{->}[u]              \\
             \underline{\rm{Gp}}(C) \ar@<1ex>[r]|-{i_{\ast}=i_{!}} \ar@<-1ex>[r] \ar@{->}[u] & \underline{\rm{Gp}}(A) \ar@<-2ex>[l]_{i^{\ast}} \ar@{->}[l]|-{i^{!}}  \ar@<1ex>[r]|-{j^{!}=j^{\ast}} \ar@<-1ex>[r] \ar@{->}[u]      & \underline{\rm{Gp}}(B) \ar@<-2ex>[l]_{j_{!}}  \ar@{->}[l]|-{j_{\ast}} \ar@{->}[u]  \\
        }
    \end{equation*}
   Assume $B$ and $C$ are virtually Gorenstein algebras. For any object $X \in (\underline{\rm{GP}}(A))^{c}$, we have the following triangle in $\underline{\rm{GP}}(A)$.
    $$
         \xymatrix@C=1em{
               j_{!}j^{!}X  \ar@{->}[r] & X \ar@{->}[r] & i_{\ast}i^{\ast} X \ar@{->}[r] &  (j_{!}j^{!}X)[1]   \\
         }
    $$
    Since $j^{!}X\in (\underline{\rm{GP}}(B))^{c}=\underline{\rm{Gp}}(B)$ and $i^{\ast}X \in (\underline{\rm{GP}}(C))^{c}=\underline{\rm{Gp}}(C)$, $j_{!}j^{!}X$ and $i_{\ast}i^{\ast} X$ are all in $\underline{\rm{Gp}}(A)$. It follows that $X$ belongs to $\underline{\rm{Gp}}(A)$. Thus $A$ is a virtually Gorenstein algebra.
   \end{proof}

    \begin{exm}
        Let $A$ be a finite dimensional algebra. Take a $M\in {\rm{mod}}A$ such that ${\rm{pd}} _{A}M<\infty$. We consider the triangular matrix algebra $T=\bigl(\begin{smallmatrix}
                                                                                 A & M \\
                                                                                 0 & k
                                                                               \end{smallmatrix}\bigr)$.
        Following \cite{Zhang2013}, $M$ is compatible. Let $I$ be the functor from ${\rm{Mod}}A$ to  ${\rm{Mod}}T$ which sends $X\in {\rm{Mod}}A$ to $\bigl(\begin{smallmatrix}
                                                                                                  X \\
                                                                                                  0
                                                                                                \end{smallmatrix}\bigr) \in {\rm{Mod}}T$. By \cite[Theorem 3.3]{Zhang2013}, we get that $I$ induced a triangle equivalence from
                                                                                                $\underline{\rm{GP}}(A)$ to $\underline{\rm{GP}}(T)$ which restricts to $\underline{\rm{Gp}}$. Then by Theorem \ref{thm1},  $A$ is virtually Gorenstein if and only if so is $T$.
   \end{exm}

\section{$n$-recollements of derived categories }
    In this section we investigate the relation between $n$-recollement of derived categories and virtually Gorenstein algebras. There is a certain kind of triangle functors between derived category, named by nonnegative functor, play an essential role in this section.

     \begin{defn}[{\cite[Definition 4.1]{Hu2017}}]
       Let $\mathcal{A}$ and $\mathcal{B}$ be two abelian categories with enough projective objects. A triangle functor $F:\mathcal{D}^{b}(\mathcal{A}) \to \mathcal{D}^{b}(\mathcal{B})$ is called {\em nonnegative} if $F$ satisfies the following conditions:
       \begin{enumerate}
         \item $F(X)$ is isomorphic to a complex with zero homology in all negative degrees for all $X\in \mathcal{A}$;
         \item $F(P)$ is isomorphic to a complex in $\mathcal{K}^{b}(\rm{P}(\mathcal{B}))$ with zero terms in all negative degrees for all $P\in \rm{P}(\mathcal{A})$.
       \end{enumerate}
     \end{defn}

     Let $\mathcal{A}$ and $\mathcal{B}$ be two abelian categories. A triangle functor $F:\mathcal{D}^{b}(\mathcal{A}) \to \mathcal{D}^{b}(\mathcal{B})$ is called {\em nonnegative up to} $shifts$ if there is a integer $n$ such that $F\circ [n]$ is nonnegative.

     \begin{lem}[{\cite[Lemma 3.4]{Chen2020}}]\label{l3.3}
        Let
            $$
            \xymatrix{ \mathcal{D}({\rm{Mod}}B) \ar@<-1ex>[r]_{G}  & \mathcal{D}({\rm{Mod}}A)  \ar@<-1ex>[l]_{F}
            }
            $$
        be an adjoint pair with both $F$ and $G$ preserving $\mathcal{K}^{b}(proj)$. Then, up to shifts, $G$ restricts to a nonnegative functor from $\mathcal{D}^{b}({\rm{Mod}}B)$ to $\mathcal{D}^{b}({\rm{Mod}}A)$.
     \end{lem}

     \begin{rmk}\label{l4.3}
           Similarly, up to shifts, the functor $G$ also restricts to a nonegative functor from $\mathcal{D}^{b}({\rm{mod}}B)$ to $\mathcal{D}^{b}({\rm{mod}}A)$.
     \end{rmk}

     The following lemma is the key to link the functor of derived category and the functor of stable categories.

    \begin{lem}[{\cite[Theorem 5.3]{Hu2017}}]\label{hp5.3}
      Let $\mathcal{A}$ and $\mathcal{B}$ be two abelian categories with enough projective objects, and let $F:\mathcal{D}^{b}(\mathcal{A})\to \mathcal{D}^{b}(\mathcal{B})$ be a nonnegative triangle functor which admits a right adjoint $G$ with $G(Q) \in \mathcal{K}^{b}(\rm{P}(\mathcal{A}))$ for all $Q\in \rm{P}(\mathcal{B})$. Then there is a commutative (up to natural isomorphism) of triangle functors.
      $$
        \xymatrix{\ar @{} [dr]
        \underline{\rm{GP}}\mathcal{A} \ar[d]^{\overline{F}} \ar@{^{(}->}[r] &  \mathcal{D}^{b}(\mathcal{A})/\mathcal{K}^{b}(\rm{P}(\mathcal{A})) \ar[d]^{F} \\
        \underline{\rm{GP}}\mathcal{B} \ar@{^{(}->}[r]        &  \mathcal{D}^{b}(\mathcal{B})/\mathcal{K}^{b}(\rm{P}(\mathcal{A}))
        }
      $$
    \end{lem}
    It is obvious that the lemma still holds if $F:\mathcal{D}^{b}(\mathcal{A})\to \mathcal{D}^{b}(\mathcal{B})$ is nonnegative up to shifts.

    We start to prove the Theorem (\textrm{II}).

    \begin{thm}\label{thm2}
        Let $A,B,$ and $C$ be finite dimensional algebras, and  $\mathcal{D}(\mathrm{Mod}A)$ admits a $n$-recollement relative to  $\mathcal{D}(\mathrm{Mod} B)$ and $\mathcal{D}(\mathrm{Mod} C)$.
        \begin{enumerate}
          \item $n=5$: if $A$ is virtually Gorenstein then so are $B$ and $C$;
          \item $n\geq 6$: $A$ is virtually Gorenstein if and only if so are $B$ and $C$.
        \end{enumerate}
    \end{thm}

    \begin{proof}
    (1) Let
        \begin{equation}\label{T1}
        \xymatrix@C=3em{
            \mathcal{D}({\rm{Mod}}B)\ar@<2ex>[r]  \ar@{->}[r] \ar@<-2ex>[r]  & \mathcal{D}({\rm{Mod}}A) \ar@<-3ex>[l]  \ar@<-1ex>[l] \ar@<1ex>[l] \ar@<3ex>[l] \ar@<2ex>[r] \ar@{->}[r] \ar@<-2ex>[r]  & \mathcal{D}({\rm{Mod}}C) \ar@<-3ex>[l]  \ar@<-1ex>[l] \ar@<1ex>[l] \ar@<3ex>[l] \\
            }
        \end{equation}
      be a 5-recollement. By Lemma \ref{P2.3}, functors in the middle five rows of diagram \ref{T1} restrict to $\mathcal{D}^{b}(\rm{Mod})$, thus we have the following 3-recollement.
        \begin{equation}\label{T2}
         \xymatrix@C=3em{
                \mathcal{D}^{b}({\rm{Mod}}C)\ar@<1ex>[r]  \ar@<-1ex>[r] & \mathcal{D}^{b}({\rm{Mod}}A) \ar@<-2ex>[l]  \ar@{->}[l]  \ar@<2ex>[l] \ar@<1ex>[r] \ar@<-1ex>[r] & \mathcal{D}^{b}({\rm{Mod}}B) \ar@<-2ex>[l]  \ar@{->}[l]  \ar@<2ex>[l] \\
         }
        \end{equation}
      Also by Lemma \ref{P2.3}, functors in upper five rows of diagram \ref{T1} restrict to $\mathcal{K}^{b}(\rm{proj})$, then we have the the following 3-recollement.
        \begin{equation}\label{T3}
            \xymatrix@C=3em{
               \mathcal{K}^{b}({\rm{proj}}B)\ar@<1ex>[r]  \ar@<-1ex>[r] & \mathcal{K}^{b}({\rm{proj}}A) \ar@<-2ex>[l]  \ar@{->}[l]  \ar@<2ex>[l] \ar@<1ex>[r] \ar@<-1ex>[r] & \mathcal{K}^{b}({\rm{proj}}C) \ar@<-2ex>[l]  \ar@{->}[l]  \ar@<2ex>[l] \\
         }
        \end{equation}

      By Lemma \ref{l3.3}, functors in the upper four rows of diagram \ref{T2} are nonnegative up to shifts.

      By Lemma \ref{P2.3}, functors in the upper five rows of diagram \ref{T2} restrict to  $\mathcal{K}^{b}(\rm{Proj})$, then we have the following 2-recollement.
            \begin{equation}\label{T4}
            \xymatrix@C=3em{
                    \mathcal{K}^{b}({\rm{Proj}}C)\ar@<1ex>[r]  \ar@<-1ex>[r]   & \mathcal{K}^{b}({\rm{Proj}}A) \ar@<-2ex>[l]  \ar@{->}[l]   \ar@<1ex>[r] \ar@<-1ex>[r]  & \mathcal{K}^{b}({\rm{Proj}}B) \ar@<-2ex>[l]  \ar@{->}[l]   \\
            }
            \end{equation}

      By diagram \ref{T2}, \ref{T4} and Lemma \ref{LiuP2.5}, we have the following 2-recollement

      \begin{equation}\label{T5}
            \xymatrix@C=3em{
                    \mathcal{D}_{sg}^{b}({\rm{Mod}}C)\ar@<1ex>[r]  \ar@<-1ex>[r]   & \mathcal{D}^{b}_{sg}({\rm{Mod}}A) \ar@<-2ex>[l]  \ar@{->}[l]   \ar@<1ex>[r] \ar@<-1ex>[r]  & \mathcal{D}^{b}_{sg}({\rm{Mod}}C) \ar@<-2ex>[l]  \ar@{->}[l]   \\
            }
      \end{equation}

      By Lemma \ref{hp5.3}, we get the following recollement in which all functors are restrictions of functors in the upper three rows of diagram \ref{T5}.

             \begin{equation}\label{T6}
                \xymatrix@C=3em{
                  \underline{\rm{GP}}(C) \ar@{->}[r]   & \underline{\rm{GP}}(A) \ar@<-1ex>[l]  \ar@<1ex>[l]  \ar@{->}[r]   & \underline{\rm{GP}}(B) \ar@<-1ex>[l]  \ar@<1ex>[l]   \\
                }
            \end{equation}

      The functors in the upper two rows of diagram \ref{T6} preserve compact objects, then we have the following diagram.

        \begin{equation}\label{T7}
            \xymatrix@C=2em{
            (\underline{\rm{GP}}(C))^{c}   & (\underline{\rm{GP}}(A))^{c}  \ar@{->}[l]   & (\underline{\rm{GP}}(B))^{c} \ar@{->}[l]  \\
        }
        \end{equation}

      Analogously, functors in the middle five rows of diagram \ref{T1} restrict to $\mathcal{D}^{b}(\rm{mod})$ and functors in the upper five rows of diagram \ref{T1} restrict to $\mathcal{K}^{b}({\rm{proj}})$. By Remark \ref{l4.3}, restrictions of functors in the second to fourth rows of diagram \ref{T1} on $\mathcal{D}^{b}(\rm{mod})$ are nonnegative up to shifts. Again, by Lemma \ref{hp5.3}, we have the following recollement.

        \begin{equation}
         \xymatrix@C=3em{
                  \underline{\rm{Gp}}(C) \ar@{->}[r]    & \underline{\rm{Gp}}(A) \ar@<-1ex>[l]  \ar@<1ex>[l]  \ar@{->}[r]  & \underline{\rm{Gp}}(B) \ar@<-1ex>[l]  \ar@<1ex>[l]   \\
            }
         \end{equation}

      Above all, we have the following commutative diagram in which all vertical arrows are embeddings and all horizontal arrows are induced by functors in the second to fourth rows from up to bottom of diagram \ref{T1}.

        \begin{equation*}
            \xymatrix@C=3em{
                \underline{\rm{GP}}(C) \ar@{->}[r]|-{i_{\ast}=i_{!}}              & \underline{\rm{GP}}(A) \ar@<-1ex>[l]_{i^{\ast}} \ar@<1ex>[l]^{i^{!}}  \ar@{->}[r]|-{j^{!}=j^{\ast}}            & \underline{\rm{GP}}(B) \ar@<-1ex>[l]_{j_{!}}  \ar@<1ex>[l]^{j_{\ast}}   \\
                (\underline{\rm{GP}}(C))^{c}  \ar@{->}[u]           & (\underline{\rm{GP}}(A))^{c}  \ar@{->}[l]_{i^{\ast}} \ar@{->}[u]     & (\underline{\rm{GP}}(B))^{c} \ar@{->}[l]_{j_{!}}  \ar@{->}[u]              \\
                \underline{\rm{Gp}}(C) \ar@{->}[r]|-{i_{\ast}=i_{!}}  \ar@{->}[u] & \underline{\rm{Gp}}(A) \ar@<-1ex>[l]_{i^{\ast}} \ar@<1ex>[l]^{i^{!}}  \ar@{->}[r]|-{j^{!}=j^{\ast}} \ar@{->}[u]      & \underline{\rm{Gp}}(B) \ar@<-1ex>[l]_{j_{!}}  \ar@<1ex>[l]^{j_{\ast}} \ar@{->}[u]  \\
        }
        \end{equation*}

    By Theorem \ref{thm1}, the claim is proven.

    (2) Given a 6-recollement
     \begin{equation}\label{D1}
     \xymatrix@C=3em{
       \mathcal{D}({\rm{Mod}}B)\ar@<2ex>[r]  \ar@{->}[r] \ar@<-2ex>[r]  \ar@<-4ex>[r] & \mathcal{D}({\rm{Mod}}A) \ar@<-3ex>[l]  \ar@<-1ex>[l] \ar@<1ex>[l] \ar@<3ex>[l] \ar@<2ex>[r] \ar@{->}[r] \ar@<-2ex>[r] \ar@<-4ex>[r]  & \mathcal{D}({\rm{Mod}}C) \ar@<-3ex>[l]  \ar@<-1ex>[l] \ar@<1ex>[l] \ar@<3ex>[l] \\
         }
     \end{equation}
     with eight rows. By a similar argument as in (1), we have the following commutative diagram in which all vertical arrows are embeddings and all horizontal arrows are induced by functors in the second to fifth rows from up to bottom of diagram \ref{D1}.
    \begin{equation*}\label{D8}
      \xymatrix@C=3em{
            \underline{\rm{GP}}(C) \ar@<1ex>[r]|-{i_{\ast}=i_{!}} \ar@<-1ex>[r]              & \underline{\rm{GP}}(A) \ar@<-2ex>[l]_{i^{\ast}} \ar@{->}[l]|-{i^{!}}  \ar@<1ex>[r]|-{j^{!}=j^{\ast}} \ar@<-1ex>[r]             & \underline{\rm{GP}}(B) \ar@<-2ex>[l]_{j_{!}}  \ar@{->}[l]|-{j_{\ast}}   \\
            (\underline{\rm{GP}}(C))^{c} \ar@<-1ex>[r]_{i_{\ast}} \ar@{->}[u]           & (\underline{\rm{GP}}(A))^{c}  \ar@<-1ex>[l]_{i^{\ast}} \ar@<-1ex>[r]_{j^{!}} \ar@{->}[u]                     & (\underline{\rm{GP}}(B))^{c} \ar@<-1ex>[l]_{j_{!}}  \ar@{->}[u]              \\
             \underline{\rm{Gp}}(C) \ar@<1ex>[r]|-{i_{\ast}=i_{!}} \ar@<-1ex>[r] \ar@{->}[u] & \underline{\rm{Gp}}(A) \ar@<-2ex>[l]_{i^{\ast}} \ar@{->}[l]|-{i^{!}}  \ar@<1ex>[r]|-{j^{!}=j^{\ast}} \ar@<-1ex>[r] \ar@{->}[u]      & \underline{\rm{Gp}}(B) \ar@<-2ex>[l]_{j_{!}}  \ar@{->}[l]|-{j_{\ast}} \ar@{->}[u]  \\
      }
    \end{equation*}
    By Theorem \ref{thm1}, the claim is proven.
    \end{proof}

    As an application of Theorem \ref{thm2}, we reprove that virtually Gorensteinness is invariant under derived equivalence.
    \begin{cor}[{\cite[Theorem 8.11]{Beligiannis2005}}]
      Let $A$ and $B$ be two finite dimensional algebras and $F:\mathcal{D}({\rm{Mod}}A) \to \mathcal{D}({\rm{Mod}}B)$ be a triangle functor. If $F$ is an equivalence, then $A$ is virtually Gorenstein if and only if $B$ is virtually Gorenstein.
    \end{cor}

    \begin{proof}
       Since $F$ is an equivalence, there is a $n$-recollement of $\mathcal{D}({\rm{Mod}}B)$ relative to $\mathcal{D}({\rm{Mod}}A)$ and $0$ for $n>0$. Then by Theorem \ref{thm2}, $B$ is virtually Gorenstein if and only if $A$ is virtually Gorenstein.
    \end{proof}

\section*{Acknowledgements}
The authors thank Professor Wei Hu for providing a key reference and Yongyun Qin for helpful discussion. This work is supported by the National Natural Science Foundation of China (No.s 11901428 and 11801141).

\bibliographystyle{alpha}

\end{document}